\documentclass{article}
\usepackage[utf8]{inputenc}
\usepackage[english]{babel}
\usepackage{amsmath}
\usepackage{amsthm}
\usepackage{amssymb}
\usepackage{hyperref}
\usepackage{authblk}
\usepackage{subcaption}
\usepackage{graphicx}
\usepackage{xcolor}
\usepackage{enumitem}
\newtheorem{conj}{Conjecture}[section]
\newtheorem{theorem}{Theorem}[section]
\newtheorem{prop}[theorem]{Proposition}
\newtheorem{cor}[theorem]{Corollary}
\newtheorem{lemma}[theorem]{Lemma}
\newtheorem{example}[theorem]{Example}
\theoremstyle{definition}
\newtheorem{defn}[theorem]{Definition}
\theoremstyle{remark}

\def\mF{\mathcal F}
\def\mH{\mathcal H}
\def\mG{\mathcal G}
\def\mC{\mathcal C}

\title{A Cubical Perspective on Complements of Union-Closed Families of Sets}
\author{Dhruv Bhasin\thanks{Department of Mathematics, Indian Institute of Science Education and Research, Pune\\Email id: bhasin.dhruv@students.iiserpune.ac.in} }

\begin{document}

\maketitle
\begin{abstract}
    Complements of union-closed families of sets, over a finite ground set, are known as simply rooted families of sets. Cubical sets are widely studied topological objects having applications in computational homology. In this paper, we look at simply rooted families of sets from the perspective of cubical sets. That is, for every family $\mathcal F$ of subsets of a finite set, we construct a natural cubical set $X(\mathcal F)$ (corresponding to it). We show that for every simply rooted family $\mathcal F$, containing the empty set, the cubical set $X(\mathcal F)$ is always acyclic (that is, it has trivial reduced cubical homology). As a consequence of this, using the Euler-Poincar\`{e} formula, we obtain a formula satisfied by all simply rooted families of sets which contain the empty set. We also provide an elementary proof of this formula. 
\end{abstract}
\section{Introduction}
Intersection-closed structures are ubiquitous in all of mathematics. Collections of all subgroups of a group, subspaces of a vector space, subrings of a ring, independents sets of a graph are closed under intersections and the list goes on. In combinatorics, the dual concept of intersection-closed families of sets, namely union-closed families of sets is very widely studied. Most of this study is driven by the quest of solving the famous union-closed sets conjecture due to Frankl (see \cite{frankl1995extremal}). Let $[n] = \{1, \dots, n\}$ and let $2^{[n]}$ denote its power set. A family of sets $\mF \subseteq 2^{[n]}$ is said to be \textit{union-closed} if for every $A, B \in \mF$, we have $A \cup B \in \mF$. Frankl's union-closed sets conjecture states that:

\begin{conj}\label{conj:frankl}
    (see \cite{frankl1995extremal}) Let $\mathcal F \subseteq 2^{[n]}$ be a union-closed family of sets such that $\mF \neq \{\emptyset\}$. Then, there is an element $i \in [n]$ such that $i$ is in at least half of the member sets of $\mF$.
\end{conj}

Conjecture~\ref{conj:frankl} has been studied in various contexts. It has formulations in the languages of Lattice Theory (\cite{poonen1992union}) and Graph Theory (\cite{bruhn2015graph}). For a survey of the union-closed sets conjecture, we refer the reader to \cite{bruhn2015journey}. Recently, information theory based methods were used (in \cite{gilmer2022constant}) for proving the first constant lower bound for Conjecture~\ref{conj:frankl}. That is, in \cite{gilmer2022constant}, the author showed that for every union-closed family of sets with at least two elements, there is an element in at least $1\%$ of the member sets of the family. Subsequently his ideas were improved to show that there is an element in at least $38.24\%$ many member sets of the family (see \cite{alweiss2022improved}, \cite{cambie2022better}, \cite{chase2022approximate}, \cite{pebody2022extension}, \cite{sawin2022improved}, \cite{yu2023dimension}). 

The relation between union-closed families of sets and simply rooted families was first made by Balla, Bollob\'{a}s, Eccles (see \cite{balla2013union}). In this paper, the authors investigated Conjecture~\ref{conj:frankl} for large union-closed families of sets. That is, they showed that Conjecture~\ref{conj:frankl} is true for all union-closed families of sets $\mathcal G \subseteq 2^{[n]}$ satisfying $|\mathcal G| \geq \frac{2}{3}2^n$. The authors then introduced and investigated simply rooted families for a slight strengthening of their result. In \cite{eccles2016stability}, the author further investigated simply rooted families to show that Conjecture~\ref{conj:frankl} holds for union-closed families $\mathcal G \subseteq 2^{[n]}$ satisfying $|\mathcal G| \geq (\frac{2}{3} - \frac{1}{104})2^n$. In \cite{karpas2017two}, the author studied simply rooted families using tools from Boolean analysis to show that there is a constant $c > 0$ such that Conjecture~\ref{conj:frankl} holds for all union-closed families satisfying $|\mathcal F| \geq (\frac{1}{2} - c)2^n$. In \cite{studer2021asymptotic}, the author showed that the equivalent version of Conjecture~\ref{conj:frankl} for simply rooted families holds asymptotically.

Informally speaking, \textit{cubical sets} are defined to be those subsets of the $n$-dimensional euclidean space which are created by putting together cubes (of dimension at most $n$) having vertices in the lattice $\mathbb Z^n$. Cubical sets have been widely studied in various contexts and both from the theoretical and practical point of view (see for example: \cite{barcelo2021homology}, \cite{choe2022cubical}, \cite{duval2011cellular}, \cite{haglund2012combination}, \cite{hetyei1994simplicial}, \cite{hetyei1995stanley}, \cite{kaczynski2004computational}, \cite{wagner2011efficient}). In this paper, for every family $\mF \subseteq 2^{[n]}$, we define a natural cubical set corresponding to it, denoted $X(\mF) \subseteq \mathbb R^n$, which is formed by geometrically putting together the `cubes' contained in $\mF$. (For a formal definition, we refer the reader to Definition~\ref{family-realization}). To the best of our knowledge, the cubical set $X(\mF)$ has not been studied in the literature. 

We ask the question: what is the homology of $X(\mF)$ when $\mF$ is simply rooted? Our main result is:

\begin{theorem}\label{thm:acyclic}
    Let $\mathcal F \subseteq 2^{[n]}$ be a simply rooted family of sets such that $\emptyset \in \mF$. Then, $X(\mF)$ is acyclic.
\end{theorem}

Using Theorem~\ref{thm:acyclic}, and the Euler-Poincar\`{e} formula, we obtain

 \begin{cor}\label{cor:euler-poincare}
     Let $\mF$ be a simply rooted family of sets such that $\emptyset \in \mF$. Let $\mathcal C_k(\mF) = \{[A, B] : A \subseteq B, [A, B] \subseteq \mathcal F, |B\setminus A| = k\}$ where $[A, B] = \{C \in 2^{[n]}: A\subseteq C \subseteq B\}$. Then, 
     \begin{equation}\label{eq:euler-poincare}
         \sum_{k=0}^n (-1)^k |\mC_k(\mF)| = 1.
     \end{equation}    
 \end{cor}

 We provide an elementary proof of Corollary~\ref{cor:euler-poincare} in Section~\ref{sec:properties-simply-rooted} using Lemma~\ref{lemma:rooted-euler-poincare}. As depicted by Lemma~\ref{lemma:rooted-euler-poincare}, Equation~\ref{eq:euler-poincare}, is the sum of $2^n$ equations each corresponding to a set $A \in 2^{[n]}$.

 While our results do not improve upon Conjecture~\ref{conj:frankl}, in this paper, we give a topological insight regarding simply rooted families of sets. Informally speaking, Theorem~\ref{thm:acyclic} says that simply rooted families, containing the empty set, are `simple in nature' from the point of view of homology. We hope that studying simply rooted families further with this perspective will lead to more insights regarding them and in turn, regarding union-closed families of sets.

 In Section~\ref{sec:basic-notions}, we define the basic terminology required for this work. In particular, we give the definition of the specific cubical set $X(\mathcal F)$ under consideration. We also give examples showing that one can not drop either assumption in Theorem~\ref{thm:acyclic}. At the end of this section, we give the idea of our proof of Theorem~\ref{thm:acyclic}. In Section~\ref{sec:properties-simply-rooted}, we prove some properties of simply-rooted families of sets needed for the proof of Theorem~\ref{thm:acyclic}. Using these properties, we give an elementary proof of Corollary~\ref{cor:euler-poincare}. In Section~\ref{sec:preliminary-lemmas}, we prove some preliminary Lemmas involving cubical sets needed for the proof of Theorem~\ref{thm:acyclic}. In this section, we show that the cubical set we associate to a given family of sets behaves well with intersections of families of sets (Lemma~\ref{lemma:cubical-realization-commutes-intersection}). We also show that a particular class of cubical sets is always acyclic (Lemma~\ref{lemma:ayclic-cube-minus-points}), which will be needed for the proof of Theorem~\ref{thm:acyclic}. Finally, in Section~\ref{sec:proof-main-theorem}, we prove Theorem~\ref{thm:acyclic}.

\section{Cubical setting}\label{sec:cubical-setting}
\subsection{Basic notions}\label{sec:basic-notions}

Let $[n] = \{1, \dots, n\}$ and $2^{[n]}$ be its power set. For $A, B \in 2^{[n]}$, we denote $[A, B] = \{C \in 2^{[n]}: A\subseteq C \subseteq B\}$. When $A = \{i\}$, we use the shorter notation $[i, B]$ to mean $[\{i\}, B]$. A family of sets $\mathcal F \subseteq 2^{[n]}$ is said to be \textit{union-closed} if for every $A, B \in \mathcal F, A \cup B \in \mathcal F$. 

\begin{defn}\label{def:simply-rooted}
    (see \cite{balla2013union}) A family $\mathcal F \subseteq 2^{[n]}$ is said to be \textit{simply rooted} if for every non-empty $A \in \mathcal F$, there is an $i \in A$ such that $[i, A] \subseteq 2^{[n]}$.
\end{defn} 

The following result relates union-closed families and simply rooted families of sets.

\begin{prop}\label{prop:union-closed-equivalent-simply-rooted}
    \textup{(see \cite{balla2013union})} Let $\mathcal F \subseteq 2^{[n]}$ be a family of sets. Then $\mathcal F$ is union-closed if and only if $2^{[n]} \setminus \mathcal F$ is simply rooted.
\end{prop}

In this paper, we will deal with simply rooted families of sets. Following \cite{kaczynski2004computational}, we define:

\begin{defn}
    (see \cite{kaczynski2004computational}) A set $A \subseteq \mathbb R^n$ is called an \textit{elementary cube} if $A = I_1 \times \dots \times I_n$ where each $I_i = [a, b]$ such that $a, b \in \mathbb Z$ and $b-a \in \{0, 1\}$. A set $X \subseteq \mathbb R^n$ is said to be a \textit{cubical set} if it is the union of a finitely many elementary cubes.
\end{defn}

\begin{defn}\label{def:cubes-of-family} Given a family of sets $\mathcal F \subseteq 2^{[n]}$, we define
$\mathcal C(\mathcal F) = \{[A, B] \subseteq 2^{[n]} : [A, B] \subseteq \mF\}$ to be the \textit{set of cubes} of $\mathcal F$. We define $\mathcal C_k(\mathcal F) = \{[A, B] \in \mathcal C(\mathcal F) : |B \setminus A| = k\}$.
\end{defn}

\begin{defn}\label{def:cubes-realization}
    For a cube $[A, B] \subseteq 2^{[n]}$, with $A \subseteq B$, we define $I_1 \times \dots \times I_n \subseteq \mathbb R^n$ to be the \textit{geometric realization} of $[A, B]$ where $I_i =  
        \begin{cases}
            \{1\} \text{ if } i \in A \\
             [0, 1] \text{ if } i\in B\setminus A \\
             \{0\} \text{ if } i \in B^c
        \end{cases}$. We denote it by $|[A, B]|$. If $A \nsubseteq B$, we define its geometric realization to be the empty subspace of $\mathbb R^n$.
\end{defn}

\begin{defn}\label{family-realization}
     For a family of sets $\mathcal F \subseteq 2^{[n]}$, we define the \textit{geometric realization of $\mathcal F$} to be the cubical set 
        $$X(\mathcal F) = \bigcup_{[A,B] \in \mathcal C(\mathcal F)} |[A, B]|\subseteq \mathbb R^n.$$

\end{defn}

\begin{figure}[htbp]
    \centering
    % First row: Image (a) and Image (b)
    \begin{subfigure}[b]{0.6\textwidth}
        \centering
        \includegraphics[width=\textwidth]{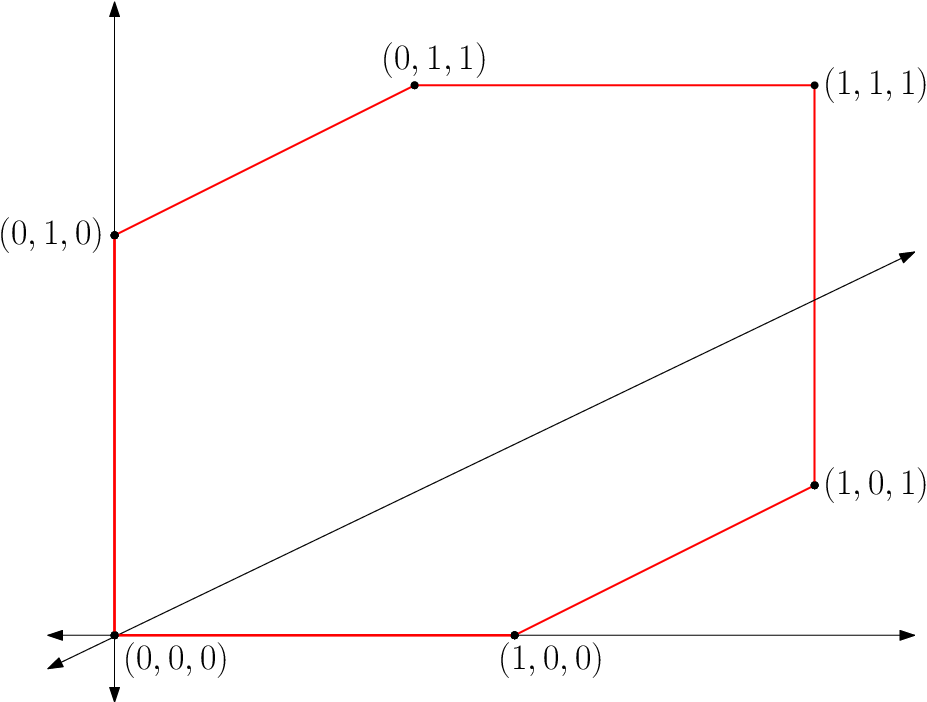} % Replace with your image path
        \caption{$X(\mF_1)$.}
        \label{fig:sub-a}
    \end{subfigure}
    \hfill
    \begin{subfigure}[b]{0.6\textwidth}
        \centering
        \includegraphics[width=\textwidth]{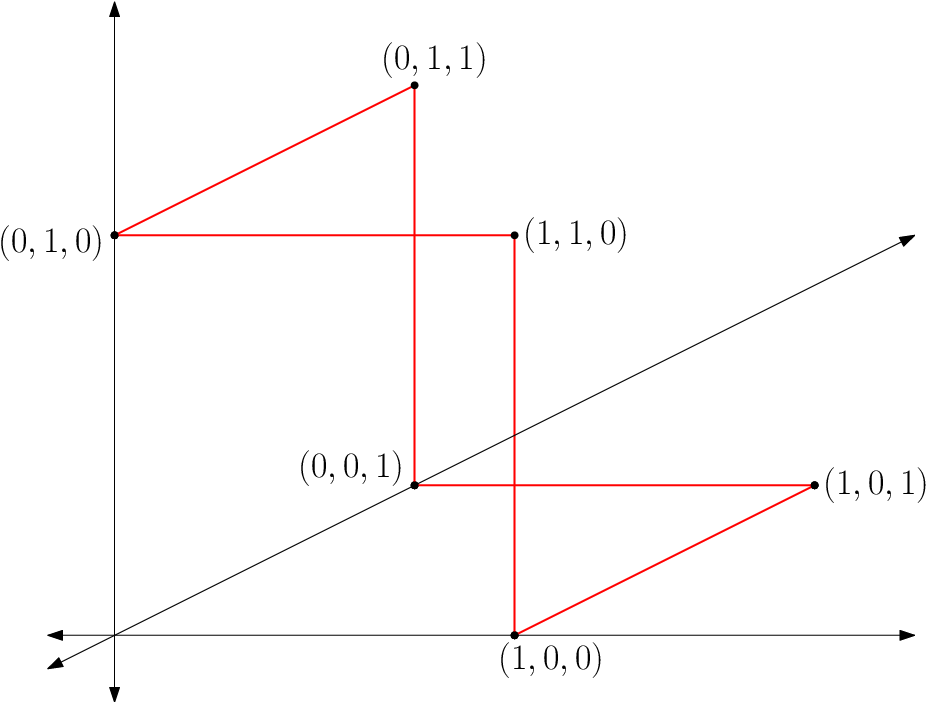} % Replace with your image path
        \caption{$X(\mF_2)$.}
        \label{fig:sub-b}
    \end{subfigure}
    
    \vskip\baselineskip % Space between the two rows
    
    % Second row: Image (c)
    \begin{subfigure}[b]{0.6\textwidth}
        \centering
        \includegraphics[width=\textwidth]{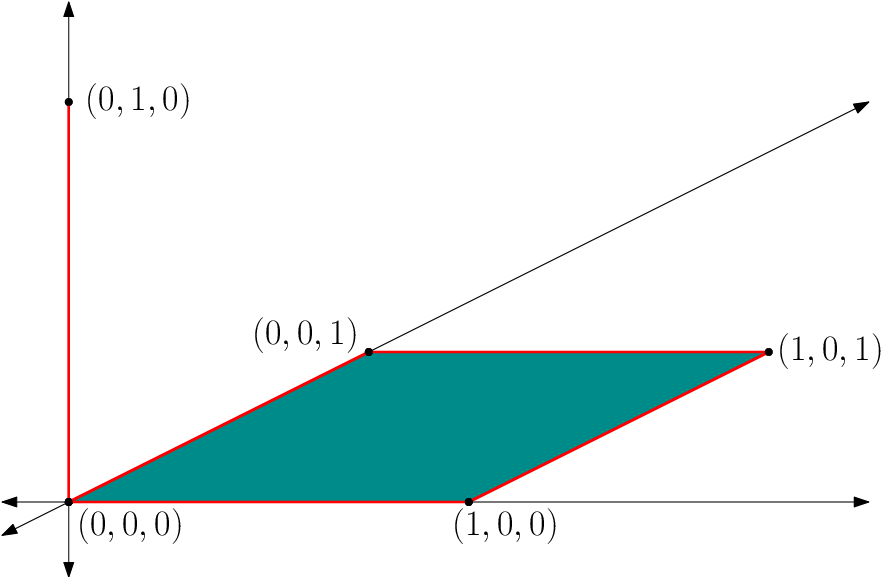} % Replace with your image path
        \caption{$X(\mF_3)$.}
        \label{fig:sub-c}
    \end{subfigure}
    
    \caption{This figure depicts $X(\mathcal F)$ for various families of sets $\mathcal F$.}
    \label{fig:main}
\end{figure}

\begin{defn}\label{def:acylic-cubical-set}
    (see \cite{kaczynski2004computational}) A cubical set $X \subseteq \mathbb R^n$ is said to be \textit{acyclic} if \begin{enumerate}
        \item $X$ is non-empty and connected,
        \item $H_i(X)$ is trivial for every $i \geq 1$, where $H_i(X)$ is $i^{\text{th}}$ cubical homology group of $X$.
    \end{enumerate}
\end{defn}

\begin{example}
    \textup{In Figure~\ref{fig:main}, we demonstrate $X(\mF)$ for various families of sets $\mF \subseteq 2^{[3]}$: 
    \begin{enumerate}
        \item In Figure~\ref{fig:sub-a}, we have $\mF_1 = \{\emptyset, \{1\}, \{2\}, \{1, 3\}, \{2, 3\}, \{1, 2, 3\}\}$ which is not a simply rooted family of sets. Note that $X(\mF_1)$ is homeomorphic to the circle $S^1$ in this case. This shows that the simply rooted condition is needed in Theorem~\ref{thm:acyclic}.
        \item In Figure~\ref{fig:sub-b}, we have $\mF_2 = \{\{1\}, \{2\}, \{3\}, \{1, 2\}, \{1, 3\}, \{2, 3\}\}$ which is a simply rooted family of sets. Note that in this case as well, $X(\mF_2)$ is homeomorphic to $S^1$. This shows that the condition $\emptyset \in \mF$ is necessary in Theorem~\ref{thm:acyclic}.
        \item In Figure~\ref{fig:sub-c}, we have $\mF_3 = \{\emptyset, \{1\}, \{2\}, \{3\}, \{1, 3\}\}$ which is a simply rooted family of sets and satisfies $\emptyset \in \mF_3$. We note that $X(\mathcal F_3)$ is acyclic in this case.
    \end{enumerate}}
\end{example}

We will use the following results from \cite{kaczynski2004computational}.

\begin{theorem}\label{thm:ayclic-by-union-of-two}
    \textup{(see \cite{kaczynski2004computational})} Assume $X, Y \subseteq \mathbb R^n$ are cubical sets. If $X, Y$ and $X \cap Y$ are acyclic, then $X \cup Y$ is acyclic.
\end{theorem}

\begin{defn}\label{def:start-shaped-cubical-set}
    (see \cite{kaczynski2004computational}) A cubical set $X \subseteq \mathbb R^n$ is \textit{star-shaped} with respect to a point $x \in \mathbb R^n$ if $X$ is the union of a finite number of elementary cubes each of which contains $x$.
\end{defn}

\begin{prop}\label{prop:star-shaped-acylic}
    \textup{(see \cite{kaczynski2004computational})} Every star-shaped set is acylic.
\end{prop}

Theorem~\ref{thm:ayclic-by-union-of-two} is the main tool we use for the proof of Theorem~\ref{thm:acyclic}. Our idea of proving Theorem~\ref{thm:acyclic} is: given a simply rooted family containing the empty set, say $\mF$, and a set $A \in \mF$ of maximum cardinality amongst the members of $\mF$, we write $X (\mF)  = X(\mF \setminus \{A\}) \cup X(\mF_A)$ (where $\mF_A$ is defined in Definition~\ref{def:family-of-sets-at-A}). We assume that $X(\mF \setminus \{A\})$ is acyclic by the induction hypothesis. We show that $X(\mF_A)$ is star-shaped and hence it is acyclic using Proposition~\ref{prop:star-shaped-acylic}. We show that the intersection of these two cubical sets is one among the class of cubical sets we show are acyclic in Lemma~\ref{lemma:ayclic-cube-minus-points}. Then, we are done using Theorem~\ref{thm:ayclic-by-union-of-two}.

\subsection{Some Properties of Simply-Rooted Families of Sets}\label{sec:properties-simply-rooted}

In this section, we explore some properties of simply rooted families of sets. Using these properties, we give an elementary proof of Corollary~\ref{cor:euler-poincare}. Using Corollary~\ref{cor:euler-poincare}, we also give the Euler-characteristic of $X(\mF)$ where $\mF$ is a simply rooted family of sets such that $\emptyset \notin \mF$.  We note here that among the results of this section, we will only need Proposition~\ref{prop:family-of-sets-at-A-equivalent} for the proof of Theorem~\ref{thm:acyclic}.

\begin{defn}\label{def:adjoint-of-inclusion}
    Let $\mathcal F \subseteq 2^{[n]}$ be a simply rooted family of sets. We define $\phi : \mathcal F \rightarrow (2^{[n]} \setminus \mathcal F \cup \{ \emptyset\})$ by
    $$\phi : A \mapsto \bigcup_{B \in (2^{[n]} \setminus \mathcal F) \cap [\emptyset, A]} B.$$
\end{defn}

We note that $\phi$ is well-defined because $2^{[n]} \setminus \mathcal F$ is union-closed.

\begin{prop}\label{prop:roots-equivalent-complement-image-adjoint}
    Let $\mathcal F \subseteq 2^{[n]}$ be a simply rooted family of sets. Let $A \in \mathcal F$ be a non-empty set. Then,
    $$\{ i \in A : [i, A] \subseteq \mathcal F\} = A \setminus \phi(A).$$
\end{prop}
\begin{proof}
    Let $[i, A] \subseteq \mathcal F$. Let $B \in (2^{[n]} \setminus \mathcal F) \cap [\emptyset, A]$. Since $[i, A] \subseteq \mathcal F$, we see that $i \notin B$. Thus, we have $i \notin \phi(A)$ and consequently, $ i \in A \setminus \phi(A)$. We conclude that $\{i \in A : [i, A] \subseteq \mathcal F\} \subseteq A \setminus \phi(A)$.
    
    On the other hand, suppose that $i \in A \setminus \phi(A)$. Let $B \in [i, A]$. Suppose that $B \in 2^{[n]} \setminus \mathcal F$. This means that $B \subseteq \phi(A)$ and hence, $ i \in \phi(A)$. This leads to a contradiction. Consequently, we obtain that $B \in \mathcal F$. This allows us to conclude that $\{i \in A: [i, A] \subseteq \mathcal F\} \supseteq A \setminus \phi(A)$. This completes the proof.
\end{proof}

\begin{defn}\label{def:family-of-sets-at-A}
    Let $\mathcal F \subseteq 2^{[n]}$ be a family of sets. Let $A \in \mathcal F$. We define $\mathcal F_A = \{B \in \mathcal F : [B, A] \subseteq \mathcal F\}$.
\end{defn}

\begin{prop}\label{prop:family-of-sets-at-A-equivalent}
    Let $\mathcal F \subseteq 2^{[n]}$ be a simply rooted family of sets such that $\phi \in \mathcal F$. Let $A \in \mathcal F$ such that $\phi(A) \neq \emptyset$. Then, 
    $$\mathcal F_A = \bigcup_{[i, A] \subseteq \mF} [i, A].$$
\end{prop}
\begin{proof}
    It is obvious that $\bigcup_{[i, A] \subseteq \mF} [i, A] \subseteq \mF_A$. On the other hand, let $C \in \mF_A$, that is, $[C, A] \subseteq \mF$. We need to show that there is a $j \in \{i \in A : [i, A] \subseteq \mF\}$ such that $j \in C$. Suppose that, on the contrary, there is no such $j$. By Proposition~\ref{prop:roots-equivalent-complement-image-adjoint}, this means that $C \subseteq \phi(A)$. Since, $[C, A] \subseteq \mF$, we obtain that $\phi(A) \in \mF$, which is a contradiction. Thus, we obtain that $\mF_A \subseteq \bigcup_{[i, A] \subseteq \mF} [i, A]$, as required.
\end{proof}

\begin{lemma}\label{lemma:rooted-euler-poincare}
    Let $\mF$ be a simply rooted family of sets such that $\emptyset \in \mF$. Let $A \in \mathcal F$ be non-empty. Let $\mathcal C_k(\mathcal F, A) = \{[C, D] \in \mathcal C_k(\mathcal F) : D = A\}$. Then, we have,
    \begin{equation}\label{eq:rooted-euler-poincare}
        \sum_{k=0}^{|A|} (-1)^k |\mC_k(\mF, A)| = 0.
    \end{equation}
\end{lemma}
\begin{proof}
    We note that, by definition, the set $\mC_k(\mF, A)$ is bijective to the set $\{B \in \mF_A : |B| = |A| - k\}$. By Proposition~\ref{prop:family-of-sets-at-A-equivalent} and Proposition~\ref{prop:roots-equivalent-complement-image-adjoint}, $\{B \in \mF_A : |B| = |A| - k\}$ consists all the subsets of $A$, of size $|A| - k$, that contain at least one element from $A \setminus \phi(A)$. The number of such subsets is $\binom{|A|}{|A| - k} - \binom{|\phi(A)|}{|A| - k}$. Consequently, we obtain:
    \begin{align*}
        \sum_{k=0}^{|A|} (-1)^k |\mC_k(\mF, A)| &= \sum_{k=0}^{|A|} (-1)^k \left\{\binom{|A|}{|A| - k} - \binom{|\phi(A)|}{|A| - k}\right\} \\
        &= \sum_{k=0}^{|A|} (-1)^k \binom{|A|}{|A| - k} - \sum_{k=0}^{|A|} (-1)^k\binom{|\phi(A)|}{|A| - k} \\
        &= \sum_{k=0}^{|A|} (-1)^k \binom{|A|}{k} - \sum_{t=0}^{|\phi(A)|} (-1)^t\binom{|\phi(A)|}{t} \\
        &= 0.
    \end{align*}
\end{proof}

\textit{Proof of Corollary~\ref{cor:euler-poincare}.} Using Lemma~\ref{lemma:rooted-euler-poincare}, Corollary~\ref{cor:euler-poincare} follows by adding Equation~\ref{eq:rooted-euler-poincare} for every non-empty $A \in 2^{[n]}$ and adding $1$ on both sides (corresponding to the empty set). $\square$ 

\begin{cor}
    Let $\mF \subseteq 2^{[n]}$ be a non-empty simply rooted family of sets such that $\emptyset \notin \mF$. For $k\geq 1$, let $c_k$ be the number of sets $A$ of $\mF$, of size $k$, such that for every $i \in A$, $[i, A] \subseteq \mF$ and let $c_0 = 1$. Then, the Euler-characteristic of $X(\mF)$ is given by $$1 - \sum_{k=0}^n (-1)^k c_k.$$
\end{cor}
\begin{proof}
    The result follows by using Corollary~\ref{cor:euler-poincare} on $\mF \cup \{\emptyset\}$ and noting that $\mC_k(\mF) \subseteq \mC_k(\mF \cup \{\emptyset))$ and that $\mC_k(\mF \cup \{\emptyset\}) \setminus \mC_k(\mF) = \{A \in \mF: |A| = k, [\phi, A] \subseteq \mF \cup \{\emptyset\}\}$.
\end{proof}

\subsection{Preliminary Lemmas}\label{sec:preliminary-lemmas}

In this subsection we prove some preliminary results needed for the proof of Theorem~\ref{thm:acyclic}.
\begin{prop}\label{prop:cubes-commute-intersection}
    Let $\mathcal F, \mathcal G \subseteq 2^{[n]}$ be family of sets. Then $\mathcal C(\mathcal F) \cap \mathcal C(\mathcal G) = \mathcal C(\mathcal F \cap \mathcal G)$.
\end{prop}
\begin{proof}
    We note that
    \begin{align*}
        [A, B] \in \mC(\mF) \cap \mC(\mG) & \Leftrightarrow [A, B] \subseteq \mF \text{ and } [A, B] \subseteq \mG \\
        & \Leftrightarrow [A, B] \subseteq \mF \cap \mG \\
        & \Leftrightarrow [A, B] \in \mC(\mF \cap \mG).
    \end{align*}
\end{proof}

The following lemma shows that geometric realization of families of sets behaves well with taking finite intersections.

\begin{lemma}\label{lemma:cubical-realization-commutes-intersection}
    Let $\mF, \mG \subseteq 2^{[n]}$ be family of sets. Then, $X(\mF) \cap X(\mG) = X(\mF \cap \mG)$.
\end{lemma}
\begin{proof}
    We begin with $[A, B], [C, D] \subseteq 2^{[n]}$. We note that $[A, B] \cap [C, D] = [A \cup C, B \cap D]$. This is because $A \subseteq E \subseteq B$ and $C \subseteq E \subseteq D \Leftrightarrow A \cup C \subseteq E \subseteq B \cap D$. 
    We first show that
    $$|[A, B]| \cap |[C, D]| = |[A \cup C, B \cap D]|.$$
    As in Definition~\ref{def:cubes-realization}, we let $|[A, B]| = I_1 \times \dots \times I_n \subseteq \mathbb R^n$ and $|[C, D]| = J_1 \times \dots \times J_n \subseteq \mathbb R^n$. This means that 
    $$|[A, B]| \cap |[C, D]| = (I_1 \cap J_1) \times \dots \times (I_n \cap J_n).$$
    We note that if $i \in A \setminus D$ then $I_i = \{1\}$ and $J_i = \{0\}$ by construction. This readily yields $|[A, B]| \cap |[C, D]| = |[A \cup C, B \cap D]| = \emptyset$. 
    On the other hand we assume that $A\subseteq D$ and $C \subseteq B$ (by symmetry). 

    Therefore, we have 
    \begin{align*}
        I_i \cap J_i = \{1\} & \Leftrightarrow (I_i = \{1\} \text{ and } J_i = [0,1]) \text{ or } (I_i = [0, 1] \text{ and } J_i = \{1\}) \\
        & \hspace{5mm}\text{ or } (I_i = \{1\} \text{ and } J_i = \{1\}) \\
        & \Leftrightarrow i \in (A\cap D \cap C^c) \cup (B \cap A^c \cap C) \cup (A \cap C) \\
        & \Leftrightarrow i \in (A \setminus C) \cup (C \setminus A) \cup (A \cap C) \\
        & \Leftrightarrow i \in A \cup C.
    \end{align*}
    We also have 
    \begin{align*}
        I_i \cap J_i = \{0\} &\Leftrightarrow (I_i = \{0\} \text{ and } J_i = [0, 1]) \text{ or } (I_i = [0, 1] \text{ and } J_i = \{0\}) \\
        &\hspace{5mm}\text{ or } (I_i = \{0\} \text{ and } J_i = \{0\}) \\
        & \Leftrightarrow i \in (B^c \cap D \cap C^c) \cup (B \cap A^c \cap D^c) \cup (B^c \cap D^c) \\
        & \Leftrightarrow i \in (D\cap B^c) \cup (B \cap D^c) \cup (B^c \cap D^c) \\
        & \Leftrightarrow i \in (B \cap D)^c.
    \end{align*}
    Finally, we have 
    \begin{align*}
        I_i \cap J_i = [0, 1] & \Leftrightarrow (I_i = [0, 1]) \text{ and } (J_i = [0, 1]) \\
        & \Leftrightarrow i \in B\cap A^c \cap D \cap C^c \\
        & \Leftrightarrow i \in B \cap D \setminus (A \cup C).
    \end{align*}
    This means that $(I_1 \cap J_1) \times \dots \times (I_n \cap J_n) = |[A \cup C, B \cap D]|$ as required. We now come back to proving our original claim.
    Let $\mathcal F, \mathcal G \subseteq 2^{[n]}$ be family of sets. We have
    \begin{align}
        X(\mF) \cap X(\mG) &= (\bigcup_{[A, B] \in \mC(\mF)}|[A, B]|) \cap (\bigcup_{[C, D] \in \mC(\mG)} |[C, D]|)  \nonumber \\ 
        &= \bigcup_{[A, B] \in \mC(\mF), [C, D] \in \mC(\mG)} |[A, B]| \cap |[C, D]| \nonumber \\
        &= \bigcup_{[A, B] \in \mC(\mF), [C, D] \in \mC(\mG)} |[A \cup C, B \cap D]| \label{eq:for-lemma-cubical-commutes-intersetcion}\\
        &\supseteq \bigcup_{[A, B] \in \mC(\mF) \cap \mC(\mG)} |[A, B]|  \nonumber \\
        &= \bigcup_{[A, B] \in \mC(\mF \cap \mG)} |[A, B]| \hspace{27mm} (\text{Using Proposition~\ref{prop:cubes-commute-intersection}}) \nonumber\\
        &= X(\mF \cap \mG). \nonumber
    \end{align}
    On the other hand, we note that given $[A, B] \in \mC(\mF)$ and $[C, D] \in \mC(\mG)$, we have $[A \cup C, B \cap D] \in \mC(\mF \cap \mG)$. This is because, $[A, B] \subseteq \mF \Rightarrow [A \cup C, B \cap D] \subseteq \mF$ and $[C, D] \subseteq \mG \Rightarrow [A \cup C, B \cap D] \subseteq \mG$. Using Equation~\ref{eq:for-lemma-cubical-commutes-intersetcion}, we conclude that $X(\mF \cap \mG) \subseteq X(\mF \cap \mG)$. This completes the proof.
\end{proof}

The following Lemma tells that a specific type of cubical sets are always acylic. We will encounter these cubical sets in the proof of Theorem~\ref{thm:acyclic}.

\begin{lemma}\label{lemma:ayclic-cube-minus-points}
    Let $n \geq 2$ and $k \geq 1$ be integers such that $k < n$. Then,  
    $$\bigcup_{i=1}^k \bigcup_{j \neq i} |[i, [n] \setminus \{j\}]|$$
    is acyclic.
\end{lemma}
\begin{proof}
    We prove this result by induction on $n$. First of all, we note that for $n = 2$, we only need to check for the case when $k = 1$. The corresponding set is $|\{1\}|$ which is clearly acyclic. Suppose that the result is true for all $n \leq t$ for some $t \geq 2$. Let $n = t + 1$. We now perform induction on $k$. Note that if $k  = 1$ then the corresponding cubical set is $\bigcup_{j \neq 1} |[1, [t+1]\setminus\{j\}]|$. This is a star-shaped cubical set and hence it is acyclic by using Proposition~\ref{prop:star-shaped-acylic}. 

    Suppose that the result is true for every $k \leq r$ for some $r$ satisfying $1 \leq r < t$. Let us now consider $k = r + 1$. We have 
    \begin{align*}
        \bigcup_{i=1}^{r+1}\bigcup_{j\neq i} |[i, [t+1] \setminus \{j\}]| &= \bigg(\bigcup_{i=1}^r \bigcup_{j \neq i} |[i, [t+1] \setminus \{j\}]|\bigg) \cup \bigcup_{j \neq r+1} |[r+1, [t+1] \setminus \{j\}]| \\
        &= X_1 \cup X_2
    \end{align*}
    where we put $X_1 = \bigcup_{i=1}^t \bigcup_{j \neq i} |[i, [t+1] \setminus \{j\}]|$ and $X_2 = \bigcup_{j \neq r+1} |[r+1, [t+1] \setminus \{j\}]|$. We note that $X_1$ is acyclic by the induction hypothesis. Since $X_2$ is star-shaped, using Proposition~\ref{prop:star-shaped-acylic}, we conclude that $X_2$ is acyclic. It is easy to see that $X_1 = X(\mG)$ where $\mG = \{A \subsetneq [t+1] : A \cap [r] \neq \emptyset\}$ and $X_2 = X(\mH)$ where $\mH = \{A \subsetneq [t + 1] : r + 1 \in A \}$. This means that $X_1 \cap X_2 = X(\mG) \cap X(\mH) = X(\mG \cap \mH)$ (using Lemma~\ref{lemma:cubical-realization-commutes-intersection}). 

    Now, we note that 
    $$\mG \cap \mH = \{A \subsetneq [t+1] : r + 1 \in A, A \cap [r] \neq \emptyset\}.$$
    Let $\mathcal K = \{A \subsetneq [t] : A \cap [r] \neq \emptyset \}$. Clearly, $X(\mG \cap \mH)$ is homeomorphic to $X(\mathcal K)$. We note that $X(\mathcal K)$ is acyclic by the induction hypothesis. Hence, $X(\mathcal G \cap \mathcal H)$ is acyclic. Using Theorem~\ref{thm:ayclic-by-union-of-two}, we see that $X_1 \cup X_2 = \bigcup_{i=1}^{r+1}\bigcup_{j\neq i} |[i, [t+1] \setminus \{j\}]|$ is acylic. This completes the proof.
\end{proof}

We now come to our main result:

\section{Proof of Theorem~\ref{thm:acyclic}}\label{sec:proof-main-theorem}

We proceed by induction on $n$. For $n = 1$, the only families are $\{\emptyset\}$ and $\{\emptyset, \{1\}\}$. The result is clearly true for these families. Suppose that the result is true for all $n \leq t$. Let $n  = t+1$. We now perform induction on $m(\mF) = \max_{B \in \mF} |B|$ where $\mF \subseteq 2^{[t+1]}$ is a simply rooted family of sets such that $\emptyset \in \mF$. The result is easily seen to be true for all such $\mF$ satisfying $m(\mF) \leq 1$. Suppose that the result is true for all such $\mF$ satisfying $m(\mF) \leq r$ for some $r \geq 1$. Let $\mF \subseteq 2^{[t+1]}$ be such a family with $m(\mF) = r + 1$. Let $A$ be a set of maximum size in $\mF$. Since $m(\mF) = r+1$, we note that $|A| \geq 2$. As in Definition~\ref{def:family-of-sets-at-A}, let $\mF_A = \{B \in \mF : [B, A] \subseteq \mF\}$. First of all, if $\mF = \mF_A$ then we are done since in this case, $\emptyset \in \mF_A$ and consequently $\mF_A$ is an elementary cube and hence it is acyclic using Proposition~\ref{prop:star-shaped-acylic}. So, we assume that $\mF \neq \mF_A$. We note that $\mF  = (\mF \setminus \{A\}) \cup \mF_A$. We also note that $\mC(\mF) = \mC(\mF \setminus \{A\}) \cup \mC(\mF_A)$. To see this, we begin with $[C, D] \in \mathcal C(\mathcal F)$. If $D = A$ then $[C, D] \in \mC(\mF_A)$. On the other hand, if $D \neq A$, we have $[C, D] \in \mC(\mF\setminus \{A\})$ (because $A$ is of maximum cardinality). This allows us to conclude that $\mathcal C(\mathcal F) \subseteq \mathcal C(\mathcal F \setminus \{A\}) \cup \mathcal C(\mathcal F_A)$. On the other hand, $\mC(\mF) \supseteq \mC(\mF \setminus \{A\}) \cup \mC(\mF_A)$ is obviously true. This shows that $\mC(\mF) = \mC(\mF \setminus \{A\}) \cup \mC(\mF_A)$. Now, 
    \begin{align*}
        X(\mF \setminus \{A\}) \cup X(\mF_A) &= \bigg(\bigcup_{[C, D] \in \mC(\mF \setminus \{A\}} |[C, D]|\bigg) \cup \bigg(\bigcup_{[C, D] \in \mC(\mF_A)} |[C, D]|\bigg) \\
        &= \bigcup_{[C, D] \in \mC(\mF \setminus \{A\}) \cup \mC(\mF_A)} |[C, D]| \\
        &= \bigcup_{[C, D] \in \mC(\mF)} |[C, D]| \\
        &= X(\mF).
    \end{align*}
    This shows that $X(\mF) = X(\mF \setminus \{A\}) \cup X(\mF_A)$. Now, we know that $X(\mF\setminus \{A\})$ is acyclic by the induction hypothesis. Using Proposition~\ref{prop:family-of-sets-at-A-equivalent}, we have that $\mF_A = \bigcup_{[i, A] \subseteq \mF} [i, A]$. We note that $\{[i, A] \subseteq \mF_A\}$ forms a maximal set of cubes of $\mF_A$. To see this, note that if $[C, D] \subseteq \mF$ then, since $C \in \mF_A$ there is an $i \in A$ such that $[i, A] \subseteq \mF$ and $C \in [i, A]$. Since, $D \subseteq A$, we obtain that $[C, D] \subseteq [i, A]$. Thus, we obtain that $X(\mF_A) = \bigcup_{[i, A] \subseteq \mF} |[i, A]|$ and hence, it is star-shaped. Using Proposition~\ref{prop:star-shaped-acylic}, we obtain that it is acyclic. Now, $X(\mF \setminus \{A\}) \cap X(\mF_A) = X((\mF\setminus \{A\}) \cap \mF_A) = X(\mF_A \setminus \{A\})$ using Lemma~\ref{lemma:cubical-realization-commutes-intersection}. Since $|A| \geq 2$, there exists an $i$ such that $[i, A] \subseteq \mF$ and hence, $\mF_A \setminus \{A\} \neq \emptyset$. This shows that $X(\mF_A \setminus \{A\}) \neq \emptyset.$ Thus, if we can show that $X(\mF_A \setminus \{A\})$ is acyclic then we will be done using Theorem~\ref{thm:ayclic-by-union-of-two}. We consider the following two cases:

    \textbf{Case 1:} Suppose that for every $i \in A$, we have that $[i, A] \subseteq \mF$. Since $\emptyset \in \mF$, this clearly means that $[\emptyset, A] \in \mF$. In this case $X(\mF \setminus \{A\}) = \bigcup_{i\in A} |[\emptyset, A\setminus \{i\}]|$ which is a star-shaped set and hence acyclic using Proposition~\ref{prop:star-shaped-acylic}. 

    \textbf{Case 2:} Suppose that there is an $i \in A$ such that $[i, A] \nsubseteq \mF$. Using $\mF_A = \bigcup_{[i, A]\subseteq \mF} [i,A]$, we obtain that $\mF_A \setminus \{A\} = \bigcup_{[i, A] \subseteq \mF} \bigcup_{j \neq i, j \in A} [i, A \setminus \{j\}]$. Since, $\{[i, A \setminus \{j\}]: [i, A] \subseteq \mF, j \in A, i \neq j\}$ forms a maximal set of cubes of $\mathcal F_A \setminus \{A\}$, we obtain that  $X(\mF_A \setminus \{A\}) = \bigcup_{[i, A] \subseteq \mF} \bigcup_{j \neq i, j \in A}|[i, A \setminus \{j\}]|$. Since, $|\{i \in A : [i, A] \subseteq \mF\}| < |A|$, using Lemma~\ref{lemma:ayclic-cube-minus-points}, it follows that $X(\mF_A \setminus \{A\})$ is acyclic. 
    
    This completes the proof.
\bibliographystyle{plain}
\bibliography{cubical_uc}

\begin{thebibliography}{10}

\bibitem{alweiss2022improved}
Ryan Alweiss, Brice Huang, and Mark Sellke.
\newblock Improved lower bound for frankl's union-closed sets conjecture.
\newblock {\em arXiv preprint arXiv:2211.11731}, 2022.

\bibitem{balla2013union}
Igor Balla, B{\'e}la Bollob{\'a}s, and Tom Eccles.
\newblock Union-closed families of sets.
\newblock {\em Journal of Combinatorial Theory, Series A}, 120(3):531--544, 2013.

\bibitem{barcelo2021homology}
H{\'e}l{\`e}ne Barcelo, Curtis Greene, Abdul~Salam Jarrah, and Volkmar Welker.
\newblock Homology groups of cubical sets with connections.
\newblock {\em Applied Categorical Structures}, 29:415--429, 2021.

\bibitem{bruhn2015graph}
Henning Bruhn, Pierre Charbit, Oliver Schaudt, and Jan~Arne Telle.
\newblock The graph formulation of the union-closed sets conjecture.
\newblock {\em European Journal of Combinatorics}, 43:210--219, 2015.

\bibitem{bruhn2015journey}
Henning Bruhn and Oliver Schaudt.
\newblock The journey of the union-closed sets conjecture.
\newblock {\em Graphs and Combinatorics}, 31:2043--2074, 2015.

\bibitem{cambie2022better}
Stijn Cambie.
\newblock Better bounds for the union-closed sets conjecture using the entropy approach.
\newblock {\em arXiv preprint arXiv:2212.12500}, 2022.

\bibitem{chase2022approximate}
Zachary Chase and Shachar Lovett.
\newblock Approximate union closed conjecture.
\newblock {\em arXiv preprint arXiv:2211.11689}, 2022.

\bibitem{choe2022cubical}
Seungho Choe and Sheela Ramanna.
\newblock Cubical homology-based machine learning: An application in image classification.
\newblock {\em Axioms}, 11(3):112, 2022.

\bibitem{duval2011cellular}
Art~M Duval, Caroline~J Klivans, and Jeremy~L Martin.
\newblock Cellular spanning trees and laplacians of cubical complexes.
\newblock {\em Advances in Applied Mathematics}, 46(1-4):247--274, 2011.

\bibitem{eccles2016stability}
Tom Eccles.
\newblock A stability result for the union-closed size problem.
\newblock {\em Combinatorics, Probability and Computing}, 25(3):399--418, 2016.

\bibitem{frankl1995extremal}
P.~Frankl.
\newblock Extremal set systems.
\newblock In {\em Handbook of combinatorics}, volume~2, pages 1293--1329. 1995.

\bibitem{gilmer2022constant}
Justin Gilmer.
\newblock A constant lower bound for the union-closed sets conjecture.
\newblock {\em arXiv preprint arXiv:2211.09055}, 2022.

\bibitem{haglund2012combination}
Fr{\'e}d{\'e}ric Haglund and Daniel~T Wise.
\newblock A combination theorem for special cube complexes.
\newblock {\em Annals of mathematics}, pages 1427--1482, 2012.

\bibitem{hatcher2005algebraic}
Allen Hatcher.
\newblock {\em Algebraic topology}.
\newblock Cambridge University Press, 2002.

\bibitem{hetyei1994simplicial}
Gabor Hetyei.
\newblock {\em Simplicial and cubical complexes: anologies and differences}.
\newblock PhD thesis, Massachusetts Institute of Technology, 1994.

\bibitem{hetyei1995stanley}
G{\'a}bor Hetyei.
\newblock On the stanley ring of cubical complex.
\newblock {\em Discrete \& computational geometry}, 14:305--330, 1995.

\bibitem{kaczynski2004computational}
Tomasz Kaczynski, Konstantin~Michael Mischaikow, and Marian Mrozek.
\newblock {\em Computational homology}, volume 157.
\newblock Springer, 2004.

\bibitem{kaczynski2007ideas}
Tomasz Kaczynski, Marian Mrozek, and Anik Trahan.
\newblock Ideas from zariski topology in the study of cubical homology.
\newblock {\em Canadian Journal of Mathematics}, 59(5):1008--1028, 2007.

\bibitem{karpas2017two}
Ilan Karpas.
\newblock Two results on union-closed families.
\newblock {\em arXiv preprint arXiv:1708.01434}, 2017.

\bibitem{pebody2022extension}
Luke Pebody.
\newblock Extension of a method of gilmer.
\newblock {\em arXiv preprint arXiv:2211.13139}, 2022.

\bibitem{poonen1992union}
Bjorn Poonen.
\newblock Union-closed families.
\newblock {\em Journal of Combinatorial Theory, Series A}, 59(2):253--268, 1992.

\bibitem{sawin2022improved}
Will Sawin.
\newblock An improved lower bound for the union-closed set conjecture.
\newblock {\em arXiv preprint arXiv:2211.11504}, 2022.

\bibitem{studer2021asymptotic}
Luca Studer.
\newblock An asymptotic version of frankl’s conjecture.
\newblock {\em The American Mathematical Monthly}, 128(7):652--654, 2021.

\bibitem{wagner2011efficient}
Hubert Wagner, Chao Chen, and Erald Vu{\c{c}}ini.
\newblock Efficient computation of persistent homology for cubical data.
\newblock In {\em Topological methods in data analysis and visualization II: theory, algorithms, and applications}, pages 91--106. Springer, 2011.

\bibitem{yu2023dimension}
Lei Yu.
\newblock Dimension-free bounds for the union-closed sets conjecture.
\newblock {\em Entropy}, 25(5):767, 2023.

\end{thebibliography}
\nocite{*}

\end{document}